\documentclass{amsart}

\usepackage{amssymb,verbatim}
\usepackage{amsmath,amsfonts}
\usepackage[mathscr]{euscript} % make a nice \Net
\usepackage{amsthm}
\usepackage{color}
\usepackage{url}
\usepackage{graphicx} %to include figure
\usepackage{enumitem} %to use nice enumeration with changed indent
\usepackage{lineno,xcolor} %to make line numbers
\usepackage{hyperref} % to make links for lemmas and theorems
\hypersetup{colorlinks} %so the links look nicer
\usepackage{bbm} % for the characteristic function 1
\usepackage{mathrsfs} %to use \mathscr

\usepackage[lined,algonl,boxed,norelsize]{algorithm2e} % to make algorithm
\usepackage{diagbox}%making diagonal for tables
\usepackage{threeparttable} %making footnotes in table
\usepackage{extarrows} %to use \xlongrightarrow, note that it doesn't come with \xdashrightarrow! We need to define it ourselves.
\usepackage{tikz-cd} %to draw commutative diagram
\usepackage{ltablex} % to allow for automatic new line and multipages in table 

\setlist[enumerate]{
	label=\textnormal{({\roman*})},
	ref={\roman*}}

\def\.{\hskip.06cm}

\makeatletter
\newtheorem*{rep@theorem}{\rep@title}
\newcommand{\newreptheorem}[2]{%
	\newenvironment{rep#1}[1]{%
		\def\rep@title{#2 \ref{##1}}%
		\begin{rep@theorem}}%
		{\end{rep@theorem}}}
\makeatother

\newreptheorem{theorem}{Theorem}
\newreptheorem{corollary}{Corollary}

\newtheorem{theorem}{Theorem}[section]
\newtheorem{lemma}[theorem]{Lemma}

\theoremstyle{definition}
\newtheorem{definition}[theorem]{Definition}

\newtheorem{problem}[theorem]{Problem}

\theoremstyle{remark}

\numberwithin{equation}{section}

\makeatletter
\newcommand*{\da@rightarrow}{\mathchar"0\hexnumber@\symAMSa 4B }
\newcommand*{\da@leftarrow}{\mathchar"0\hexnumber@\symAMSa 4C }
\newcommand*{\xdashrightarrow}[2][]{%
	\mathrel{%
		\mathpalette{\da@xarrow{#1}{#2}{}\da@rightarrow{\,}{}}{}%
	}%
}
\newcommand{\xdashleftarrow}[2][]{%
	\mathrel{%
		\mathpalette{\da@xarrow{#1}{#2}\da@leftarrow{}{}{\,}}{}%
	}%
}
\newcommand*{\da@xarrow}[7]{%
	\sbox0{$\ifx#7\scriptstyle\scriptscriptstyle\else\scriptstyle\fi#5#1#6\m@th$}%
	\sbox2{$\ifx#7\scriptstyle\scriptscriptstyle\else\scriptstyle\fi#5#2#6\m@th$}%
	\sbox4{$#7\dabar@\m@th$}%
	\dimen@=\wd0 %
	\ifdim\wd2 >\dimen@
	\dimen@=\wd2 %   
	\fi
	\count@=2 %
	\def\da@bars{\dabar@\dabar@}%
	\@whiledim\count@\wd4<\dimen@\do{%
		\advance\count@\@ne
		\expandafter\def\expandafter\da@bars\expandafter{%
			\da@bars
			\dabar@ 
		}%
	}%  
	\mathrel{#3}%
	\mathrel{%   
		\mathop{\da@bars}\limits
		\ifx\\#1\\%
		\else
		_{\copy0}%
		\fi
		\ifx\\#2\\%
		\else
		^{\copy2}%
		\fi
	}%   
	\mathrel{#4}%
}
\makeatother

\makeatletter % \overrightharpoon is like \overrightarrow but with a harpoon in place of the arrow
\newcommand{\overrightharpoon}{%
	\mathpalette{\overarrow@\rightharpoonfill@}}
\def\rightharpoonfill@{\arrowfill@\relbar\relbar\rightharpoonup}
\makeatother

\makeatletter % \overrightsmallharpoon has a smaller harpoon than \overrightharpoon
\newcommand{\osh}{\mathpalette{\overarrowsmall@\rightharpoonfill@}}
\def\rightharpoonfill@{\arrowfill@\relbar\relbar\rightharpoonup}
\newcommand{\overarrowsmall@}[3]{%
	\vbox{%
		\ialign{%
			##\crcr
			#1{\smaller@style{#2}}\crcr
			\noalign{\nointerlineskip}%
			$\m@th\hfil#2#3\hfil$\crcr
		}%
	}%
}
\def\smaller@style#1{%
	\ifx#1\displaystyle\scriptstyle\else
	\ifx#1\textstyle\scriptstyle\else
	\scriptscriptstyle
	\fi
	\fi
}
\makeatother

\DeclareMathOperator{\Ac}{\mathcal{A}} % set of invariant spanning subgraph
\DeclareMathOperator{\Bc}{\mathcal{B}} %set of invariant rotor configurations
\DeclareMathOperator{\Cc}{\mathcal{C}} %set of rotor configurations with not many visits to origin
\DeclareMathOperator{\Dc}{\mathcal{D}} %set of  rotor configurations with long tails
\DeclareMathOperator{\Eb}{\mathbb{E}} %Expectation
\DeclareMathOperator{\Ec}{\mathcal{E}} %set of weird rotor configurations
\DeclareMathOperator{\Gc}{\mathcal{G}} %Green function
\DeclareMathOperator{\ousf}{\osh{\mathsf{USF}}} %oriented uniform spanning forest
\DeclareMathOperator{\owusf}{\osh{\mathsf{WSF}}} %wired spanning forest
\DeclareMathOperator{\Pb}{\mathbb{P}} %Probability
\DeclareMathOperator{\SF}{\osh{\textnormal{SF}}} %the set of oriented spanning forests
\DeclareMathOperator{\Zb}{\mathbb{Z}} % integers
\begin{document}

\title[Rotor walk stationarity]{Infinite-step stationarity of rotor walk and the wired spanning forest}

\author{Swee Hong Chan}
\address{Swee Hong Chan, Department of Mathematics, UCLA,  Los Angeles, California 90095}
%\curraddr{Department of Mathematics and Statistics,
%Case Western Reserve University, Cleveland, Ohio 43403}
\email{sweehong@math.ucla.edu}
\thanks{The author was supported in part by NSF grant DMS-1455272.}

%    General info

\date{January 6, 2020 and, in revised form, \today.}



\begin{abstract}
We study  rotor walk, a  deterministic counterpart of the simple random walk, on infinite transient graphs.
We show  that
the final rotor configuration of the rotor walk (after the walker escapes to infinity) 
follows the law of 
the  wired uniform spanning forest oriented toward infinity (OWUSF) measure
when the initial rotor configuration is sampled from OWUSF.
This result holds for all graphs for which each tree in the wired spanning forest has one single end almost surely.
This answers a question posed in a previous work of the author (Chan 2018). 
\end{abstract}

\maketitle

\section{Introduction}\label{section: main result}
%This paper features two main characters, (1) rotor walks and (2) the oriented wired spanning forest.
%We will establish that the latter is a mathematical  companion  of the former, in a manner to be made precise.

In a \emph{rotor walk} on a graph $G:=(V,E)$, each vertex of $G$ has   an arrow called a \emph{rotor} that points  toward a neighbor of the vertex;
all of the rotors together   constitute the \emph{rotor configuration} of the walk.
At each discrete time step,
the walker changes the rotor of its current location
by following a prescribed periodic sequence.
After that, the walker moves to the location to which the new rotor is pointing.
%Note that the trajectory of the walker is determined once the initial rotor configuration is given.

Continuing on the previous work of the author~\cite{Cha19}, our initial rotor configuration will be sampled from the wired spanning forest oriented toward infinity measure:
%The second main character, the \emph{wired spanning forest oriented toward infinity} $\owusf(G)$, is defined as follows.
%One major difference of this paper compared to other works in the literature is our choice of  initial rotor configuration;
%it is sampled from  the {oriented wired uniform spanning forest} measure.
Let $G$ be a simple connected graph  that is  locally finite and transient.
Let  $Z_0\supseteq Z_1 \supseteq \ldots$ be  subsets of $V$  such that each $V \setminus Z_R$ is finite and $\bigcap_{R\geq 0} Z_R=\varnothing$.
Let $G_R$ be obtained from $G$ by identifying all vertices in $Z_R$ to one new vertex $z_R$, and let $\mu_R$ be the uniform measure on spanning trees of $G_R$ oriented toward $z_R$.
The \emph{wired spanning forest oriented toward infinity} $\owusf$ is the (unique) infinite volume limit of the measure.
We will establish in this paper that this measure is an infinite-step stationary measure for rotor walk, in a manner to be made precise.
%and therefore  justifies our choice of the initial rotor configuration.

%The study of oriented wired spanning forest in the context of rotor walks was started by Chan~\cite{Cha19}, 

Our starting point is the result of 
\cite[Theorem~1.1]{Cha19} that, on average, 
the rotor walk visits each vertex at most as frequently as the simple random walk, i.e.,
% expected number of visits by the rotor walk is at most equal to that of the simple random walk.
\begin{equation}\label{equation: odometer wsf inequality}
\Eb_{\rho \sim \owusf}[u(\rho)(x)]  \quad \leq \quad  \Gc(x)   \qquad \forall\,  x \in V,   
\end{equation}
where $u(\rho)(x)$ is the number of visits to $x$ by the rotor walk with initial rotor configuration $\rho$ sampled from $\owusf$,
and $\Gc(x)$ is the expected number of visits to $x$ by the simple random walk.

It follows from \eqref{equation: odometer wsf inequality} that, when $G$ is a transient graph, 
the aforementioned rotor walk visits each vertex only finitely many times a.s..
This  implies that the sequence of rotor configurations
$(\rho_t)_{t \geq 0}$ at the $t$-th step of the walk converges pointwise (as $t \to \infty$) to a unique rotor configuration $\sigma(\rho)$, which we refer to as  the  \emph{final rotor configuration}  of the rotor walk. 

For finite graphs, it is a folk theorem that $\sigma(\rho)$ has the same law as $\rho$ when  $\rho$ is sampled from the uniform spanning tree measure~(see \cite[Lemma~3.11]{HLM08}).
We would like to prove the same result for infinite graphs.
%Previous   studies for finite graphs (see \cite[Lemma~3.11]{HLM08})
%  suggest that    $\sigma(\rho)$ should have  the same law as $\rho$.

\begin{problem}[Infinite-step stationarity]\label{problem: stationarity}
	Let  the initial rotor configuration  be sampled from $\owusf$. Show that 
	the final rotor configuration also follows the law of $\owusf$.
\end{problem}

Somewhat surprisingly, 
the answer to Problem~\ref{problem: stationarity} can depend on the underlying graph $G$.
In \cite{Cha19},
Problem~\ref{problem: stationarity} was answered positively for  when  $G$ is the perfect $b$-ary tree ($b\geq 2$),
but was answered negatively for when  $G$ is the perfect $b$-ary tree  with an infinite path attached to its root.
The case for the integer lattice $\Zb^d$ $(d\geq 3)$ was posed as an open question in \cite[Question~9.1]{Cha19}.

In this paper we answer
Problem~\ref{problem: stationarity} positively for a large family of graphs that include $\Zb^d$.
An \emph{end} in a tree is an equivalence classes of infinite paths under the  relation where  two paths are equal if they differ by only finitely many vertices.

\begin{theorem}\label{theorem: rotor walk stationarity}
	Let $G$ be a simple connected graph that is locally finite and transient, and let $\rho$ be sampled from $\owusf$.
	Suppose that
	\begin{equation}\label{equation: 1End}
	\text{Every tree in $\owusf$ has a single  end  a.s..} \tag{1End}
	\end{equation}
	Then $\sigma(\rho)$ follows the law of $\owusf$.
\end{theorem}
Transient graphs that satisfy \eqref{equation: 1End} include $\Zb^d$ with $d\geq 3$,  vertex-transitive  graphs, and graphs with reasonable isoperimetric profile~(see~\cite{LMS08,LP16}). 
%For other broader conditions that guarantee  \eqref{equation: 1End}, see \cite{LMS08}.
It is an open problem to determine if \eqref{equation: 1End} is a necessary condition for infinite-step stationarity (see Section~\ref{section: open problems}).

Our proof of Theorem~\ref{theorem: rotor walk stationarity} uses techniques developed by   J\'{a}rai and Redig in \cite{JR08} for studying recurrent sandpiles.
We study a new process 
where the rotor walk is now terminated upon hitting a fixed finite subset $W$ of $V$, and the initial rotor configuration  is now sampled from 
the wired spanning forest $\owusf(W)$ oriented toward $W$.
The addition of the termination rule is justified by \cite[Theorem~6.2]{Cha19}, and the modification of the initial rotor configuration is justified by
\cite[Lemma~7.5]{JR08}.
By coupling the original process with the new process,   we show that Theorem~\ref{theorem: rotor walk stationarity} will follow from showing that   
the component of $W$ in $\owusf(W)$ is finite.
This condition  in turn is a consequence of  \eqref{equation: 1End}~\cite[Proposition~7.11]{JR08}, and the proof is complete.

% We first reduce Theorem~\ref{theorem: rotor walk stationarity} to checking 
%a technical condition for the rotor walk that is terminated upon hitting a specified finite subset $W$ of $V$~\cite[Theorem~6.2]{88}.
%  We then exchange the initial rotor configuration with the configuration sampled from the wired spanning forest $\owusf(W)$ oriented toward $W$; this exchange  can be done as the density of $\owusf(G)$ with respect $\owusf(W)$ is strictly positive~\cite[Lemma~7.5]{JR08}.
%This exchange  reduces the  technical condition from before to checking that  the component of $W$ in $\owusf(W)$ is finite.
%The latter  turns out to be a known fact~\cite[Proposition~7.11]{JR08}, and the proof is complete.

Here we present two interesting consequences of Theorem~\ref{theorem: rotor walk stationarity}.
The first consequence is an  upgrade of \eqref{equation: odometer wsf inequality}, and can be derived as a direct corollary of Theorem~\ref{theorem: rotor walk stationarity} and \cite[Theorem~1.1]{Cha19}.

\begin{theorem}
	Let $G$ be a simple connected graph that is locally finite, transient, and satisfies \eqref{equation: 1End}.
	Then, for all $x \in V$, 
	\[\pushQED{\qed} \Eb_{\rho \sim \owusf}[u(\rho)(x)] \quad =  \quad \Gc(x).  \qedhere \popQED\]
\end{theorem}

%We need the following notation to present the second consequence.
The second consequence is a formula for the escape rate of the rotor walk:
Start the process  with an initial environment $\rho$  and with $n$ particles initially located at a fixed
vertex $a$. 
Each of these particles in turn performs rotor walk until it either returns to $a$ or
escapes to infinity.
Denote by $I(\rho,n)$ the number of particles that do not return to $a$.
The \emph{escape rate} is  $\lim_{n \to \infty} I(\rho,n)/n$, if the limit exists.
The following is a direct corollary of Theorem~\ref{theorem: rotor walk stationarity} and \cite[Theorem~1.5]{Cha19}.

\begin{theorem}
	Let $G$ be a simple connected graph that is locally finite, transient, and satisfies \eqref{equation: 1End}.
	Then, for almost every  $\rho$  sampled from $\owusf$,  
	\begin{equation}\label{equation: escape rate}
	\lim_{n \to \infty} \frac{I(\rho,n)}{n} \quad = \quad \alpha_G, 
	\end{equation} 
	where $\alpha_G$ is the probability that the simple random walk on $G$ never returns to the initial location.
	\qed
\end{theorem}
%Described in words, Theorem~\ref{theorem: rotor walk stationarity} says that the escape rate of the rotor walk is equal to the escape rate of the simple random walk, provided that the initial rotor configuration is sampled from $\owusf$.
Theorem~\ref{theorem: rotor walk stationarity} is an upgrade of \cite[Theorem~1.4]{Cha19}, which  derived the same conclusion  but only for vertex-transitive graphs.

\subsection{Related work}
Rotor walk was first studied in \cite{WLB96} as a model of mobile agents exploring a new territoty,
and in \cite{PDDK96} as an example of  self-organized criticality~\cite{BTW88}.
It was rediscovered several times by  researchers from different disciplines~(e.g., \cite{RSW98, DTW03, Pro03}).
We refer  the reader to \cite{HLM08} for an excellent introduction and a detailed history of this subject.
% by Rabani, Sinclair and
%Wanka in \cite{RSW98} as an approach to balance the workload of each local node  in multiprocessor systems, by
%Propp [Pro01] as a way to derandomize models such as internal diffusion-limited
%aggregation (IDLA) [DF91, LBG92], and by Dumitriu, Tetali, and Winkler as part
%of their analysis of a graph-based game [DTW03]

%Rotor walk is an example of abelian network~\cite{BL16}, a family of 
%  interacting particle systems that satisfy the abelian property and often exhibit self-organized criticality~\cite{BTW88}.

The wired spanning forest was first  constructed  by Pemantle~\cite{Pem91} as an infinite volume limit of the uniform spanning tree of finite graphs.
Other methods to construct the wired spanning forest include
Wilson's algorithm~(see \cite{Wil96,BLPS01}) that uses loop-erased random walks~\cite{Law80}, and the Aldous-Broder algorithm (see \cite{Bro89,Ald90,Hut18}).
We refer the reader to \cite[Section~10]{LP16} for  a detailed discussion regarding this subject.

The phenomenon in   Theorem~\ref{theorem: rotor walk stationarity} (where   the wired spanning forest is invariant under the dynamics of the particle system)
had previously been shown for  sandpile~\cite{JR08} and (one-step) random walk with local memory~\cite{CGLL18}.
These particle systems are notably  all examples of abelian networks~\cite{BL16}.
It is still unknown if
all these results are  consequences  of a  universal  principle that holds for all  abelian networks.
% 
% explains all these results 
%  
%a similar result holds for  arbitrary abelian networks.

%The phenomenon where an object related to the wired spanning forest displays stationarity (such as in Theorem~\ref{theorem: rotor walk stationarity})
%has been observed in other interacting particle systems such as sandpiles~\cite{JR08} and random walks with local memory~\cite{CGLL18}.
%All these particle systems are examples of abelian networks~\cite{BL16}, 
% and it is still unknown if 
%a similar phenomenon happens to arbitrary abelian networks.
%

%Wired spanning forest oriented toward infinity  is not the only rotor configuration for which the corresponding escape rate satisfies \eqref{equation: escape rate}.
Other rotor  configurations for which the corresponding escape rate satisfy \eqref{equation: escape rate} had been constructed for trees by Angel and Holroyd in \cite{AH11}, for $\Zb^d$ by He in \cite{He14}, and for all graphs by Chan in \cite{Cha20}.
Note that the escape rate of a rotor walk with an arbitrary initial rotor configuration is at most equal to  the escape rate of the simple random walk (see \cite[Theorem~10]{HP10}), and the former can be strictly smaller than the latter (see \cite{AH12}).

Other aspects of rotor walk
that had been studied in the literature include recurrence and transience~\cite{LL09,AH12,HMSH15}, scaling limit~\cite{LP08,LP09}, range ~\cite{FLP16,HSH20a,HSH20b}, escape rate~\cite{FGLP14}, and its performance in simulating   the simple random walk~\cite{CDST06, CS06, DF09, HP10}.

\subsection{Outline of the paper}
In Section~\ref{section: notations and definitions} we review notations and definitions.
In Section~\ref{section: tightness ouroboros} we derive the technical  lemmas that will be used to prove Theorem~\ref{theorem: rotor walk stationarity}.
In Section~\ref{section: proof of main theorem}
we prove Theorem~\ref{theorem: rotor walk stationarity}.
In Section~\ref{section: open problems} we list some open problems.

\section{Notations and definitions}\label{section: notations and definitions}
Throughout this paper, we will denote by $G:=(V,E)$  a connected simple (i.e., no loops or multiple edges) undirected graph that is locally finite (i.e., every vertex has finitely many edges).
We will always assume that $|V|$ is infinite.
We will denote by $a \in V$ the \emph{initial location} of a rotor walk.
We will denote by $Z$ the \emph{sink} of a rotor walk, which is a (possibly empty or infinite) subset of $V$.
%(These terminologies will be defined below.)

\subsection{Rotor walk}\label{subsection: rotor walks}
Each  vertex $x \in V$ is assigned a \emph{local mechanism} $\tau_x$, 
which is a bijection on the neighbors $N(x)$ of $x$.
We assume that each local mechanism has one unique orbit (i.e., $\{\tau_x^{i}(y) \mid i\geq 0\} = N(x)$ for any neighbor $y$ of $x$).
A \emph{rotor configuration} of $G$ is a function $\rho:V \to V$ such that $\rho(x)$ is a neighbor of $x$ for any $x \in V$.

%Throughout this paper,  $a \in V$ is the \emph{initial location} of a rotor walk, $Z \subseteq V$ is the \emph{sink} of a rotor walk. 

Let $a$ be a vertex of $G$, let $Z$ be a subset of $V$, and let $\rho$ be a rotor configuration of $G$.
The corresponding \emph{rotor walk} $(X_t,\rho_t)_{t \geq 0}:= (X_t(a,Z;\rho),\rho_t(a,Z;\rho))_{t \geq 0}$
is a sequence of vertices and rotor configurations  defined recursively as follows.

Define  $X_0:=a$ and $\rho_0:=\rho$;  this indicates that $a$ is the initial location of the walker, and $\rho$ is the initial rotor configuration.
At the $t$-th step of the walk, the rotor of
the current location of the walker is incremented to point to the next vertex in the cyclic order specified by its local mechanism,  and then the walker moves to the vertex specified by this new rotor, i.e.,\begin{equation}\label{equation: definition rotor walk}
\begin{split}
\rho_{t+1}(x)& \ := \ \begin{cases}
\rho_t(x) & \text{if } x\neq X_t; \\
\tau_{X_t}(\rho_t(X_t)) & \text{if } x= X_t, \end{cases}\\
X_{t+1}& \ := \ \tau_{X_t}(\rho_t(X_t)).
\end{split}
\end{equation} 
The walk is immediately terminated if the walker reaches a vertex in the {sink} $Z$.  
Note that this  rule implies that    the rotors
of  $Z$ are inconsequential to the process.
% 
%  walk does not depend on the rotors of the sink $Z$.
Therefore, we will not specify  the value of these rotors (i.e., only $\rho(x)$ for $x \in V \setminus Z$ are  given) when such practice simplifies the exposition.

%When the initial location $a$, the sink $Z$, and the initial rotor configuration $\rho$ of the rotor walk is not clear from the context,
%we will write $X_t(a,Z;\rho)$ and  $\rho_t(a,Z;\rho)$ instead of $X_t$ and $\rho_t$, respectively.

The following lemma will be used in Section~\ref{subsection: the third technical lemma}.
An \emph{oriented path} in  $\rho$ is a sequence of vertices $x_0,\ldots, x_{\ell}$ such that  $\rho(x_i)=x_{i+1}$ for any $i \in \{0,\ldots,\ell-1\}$.

\begin{lemma}\label{lemma: trace of the trajectory}
	For any $t\geq 0$ and any $i\leq t$,  there exists an oriented path in $\rho_t$ that starts at $X_i$ and ends at $X_t$.
\end{lemma}
\begin{proof}
	It suffices to prove that, for any $t\geq 0$ and any $i< t$,
	\begin{equation}\label{equation: trace of the trajectory}
	\rho_t(X_i) \in \{ X_{i+1}, X_{i+2}, \ldots, X_t \}.  
	\end{equation}
	
	We will prove \eqref{equation: trace of the trajectory} by induction on $t$.
	First note that \eqref{equation: trace of the trajectory} is vacuously true for $t=0$.
	Now suppose that $t\geq 1$; there are two possible scenarios:
	\begin{itemize}
		\item If $X_i=X_{t-1}$, then it follows from \eqref{equation: definition rotor walk} that  $\rho_t(X_i)=X_{t}$; or
		\item If  $X_i\neq X_{t-1}$, then  
		\[\rho_t(X_i) \ = \ \rho_{t-1}(X_i) \ \in \  \{X_{i+1},\ldots, X_{t-1} \},\]
		where the first equality is due to \eqref{equation: definition rotor walk}, and the second equality is due to the induction assumption.
	\end{itemize}
	In both cases we have that \eqref{equation: trace of the trajectory} is true, as desired. 
\end{proof}

A rotor walk is \emph{transient} if every vertex of $G$ is visited by the walker at most finitely many times, and is \emph{recurrent} otherwise.
One aspect of the rotor walk that we will study in this paper is the final rotor configuration of a transient walk, defined as follows.

\begin{definition}[Final rotor configuration]\label{definition: final rotor configuration}
	The \emph{final rotor configuration} $\sigma(\rho):=\sigma(a,Z;\rho)$ of a transient rotor walk is  given by 
	\begin{equation*}
	\sigma(\rho)(x) \ := \ \lim_{t \to \infty} \rho_t(x) \qquad \forall \, x \in V. \qedhere
	\end{equation*}
	%\begin{cases} 
	%\rho_{\LV(x)+1}(x) &  \text{ if } \FV(x)\geq 0;\\
	%\rho_{0}(x)   &  \text{ if } \FV(x)=-1. 
	%\end{cases} \qedhere\] 
\end{definition} 
Note that $\sigma(\rho)$ is well defined as the sequence 
$(\rho_t(x))_{t \geq 0}$ is eventually constant by the assumption that the walk is transient.

%We refer the reader to \cite{HLM08} for a more detailed introduction to  rotor walk. 

%  {\color{red}(Remember to mention why the rotor walk is transient when the simple random walk is transient).}

% The following is a corollary of Lemma~\ref{lemma: direct path to the sink}.

% The following lemma regarding final rotor configurations  will become important in \S\ref{section: tightness ouroboros}.
%
%\begin{lemma}\label{lemma: trace of the trajectory}
%If the rotor walk  reaches $Z$ in a finite time (i.e., $X_t\in Z$ for some $t \geq 0$), then, for any $i\leq t$, there exists an oriented path in $\sigma(\rho)$ that starts at $X_i$ and ends in $Z$.
%\end{lemma}
%\begin{proof}
%It follows from Lemma~\ref{lemma: direct path to the sink} that, for any $t\geq 0$, there is an oriented path in $\rho_t$ that starts at $X_i$ and ends at $X_t$.
%The lemma now follows by taking $t$ to be the smallest integer such that $X_t$ is contained in $Z$. 
%\end{proof} 

\subsection{Oriented wired spanning forest}\label{subsection: wired spanning forest}

All initial rotor configurations in this paper will be picked from   spanning forests oriented toward $Z \subseteq V$, defined as follows.

\begin{definition}[Oriented spanning forest]
	%Let $G$ be a finite connected simple graph, and let $Z$ be a nonempty subset of $V$.
	%Let $Z$ be a subset of $V$.
	A \emph{$Z$-oriented spanning forest} of $G$ is an oriented subgraph $F$ of $G$ such that
	\begin{enumerate}
		\item Every vertex in $Z$ has outdegree $0$ in $F$;
		\item Every vertex in $G \setminus Z$ has outdegree 1 in $F$; and
		\item $F$ contains no oriented cycles. \qedhere
	\end{enumerate}
	%
	%
	%for any vertex $x$ not in $Z$, there exists a unique directed path in the  subgraph 
	%that starts at $x$ and ends at $Z$.
\end{definition}

% Note that every $Z$-oriented spanning forest has no underlying cycles.
%Note that every  vertex of $G$ has outdegree $1$ in the spanning forest if the vertex is not in $Z$, and has outdegree $0$ otherwise.
Note that  each $Z$-oriented spanning forest $F$ corresponds to a rotor configuration $\rho_F$,
where for every $x \in V\setminus Z$, the state $\rho_F(x)$ is the out-neighbor of $x$ in $F$ (the state $\rho_F(x)$ for $x \in Z$ is inconsequential, as  remarked in Section~\ref{subsection: rotor walks}). 
Due to this correspondence, we will treat $\rho$ both as  a rotor configuration and as an oriented subgraph of $G$ interchangeably throughout this paper.

We denote by $\SF(Z)$ the set of $Z$-oriented spanning forests of $G$.
% Note that if $\rho$ is a rotor configuration contained in $\SF(Z)$, then the final rotor configuration $\sigma(\rho)$ of the RWLM is also contained in.
%% This is because the outdegree of every vertex does not change  during the walk, and 
% Hence $\sigma$ defines a Markov chain on $\SF(Z)$.

\begin{definition}[Oriented uniform spanning forest]\label{definition: usf}
	Suppose that $Z$ is a subset of $G$ such that $V\setminus Z$ is finite.
	The \emph{$Z$-oriented uniform spanning forest}, denoted by $\ousf(Z)$, is the uniform probability measure on $Z$-oriented spanning forests of $G$.
\end{definition}
Note that  $\ousf(Z)$ is well defined as there are only finitely many $Z$-oriented spanning forests  by the assumption that $V \setminus Z$ is finite.

Now let $Z$ be a finite subset of $V$.
This means that $V \setminus Z$ is infinite,
so Definition~\ref{definition: usf} does not apply.
In this case, the  $Z$-oriented uniform spanning forest can be defined by using the following exhaustion method. 

%Now let  $Z$ be a  finite subset of $V$ .
Throughout this paper,  $Z_0 \supseteq Z_1 \supseteq \ldots$ will be a \emph{decreasing exhaustion} of $Z$, which is an infinite sequence of decreasing subsets of $V$ such that 
\begin{equation}\label{equation: DecEx}
\text{$V\setminus Z_R$ \. is a finite set for all $R\geq 0$,  \. and  \. $\bigcap_{R\geq 0} Z_R= Z$\..} \tag{DecEx}
\end{equation}

%An \emph{exhaustion} of $G$ is a finite sequence $(W_r)_{r\geq 0}$ of increasing finite connected subsets of $V$ such that
%$\bigcup_{r \geq 0} W_r=V$.
%Let $G_r$ be the induced subgraph of $W_r$, and let   $Z_r$ be the set
%\[Z_r:=\{x \in W_r \mid  d_{G}(x, G\setminus W_r)=1\}.   \]
%That is, $Z_r$ is the set of vertices in $W_r$ that are adjacent to a vertex not in $W_r$.
%We denote by $\mu_R$ the probability measure $\ousf(Z_R)$.  

\begin{definition}[Oriented wired uniform spanning forest]
	\label{definition: wusf}
	Suppose that $G$ is a transient graph and    $Z$ is  a finite subset of $V$.
	The \emph{$Z$-oriented wired uniform spanning forest} $\owusf(Z)$ is the (unique) probability measure on oriented subgraphs of $G$ such that, for any finite subset $B$ of oriented edges of $G$, 
	\begin{equation}\label{equation: limit definition wusf} 
	\owusf(Z)[B \subseteq F] \ = \ \lim_{R \to \infty} \ousf(Z_R) [B \subseteq {F_R}],
	\end{equation}
	where $F$ is an oriented subgraph of $G$ sampled from $\owusf(Z)$, and ${F_R}$ is an oriented subgraph of $G$ sampled from $\ousf(Z_R)$.
\end{definition}
When the sink $Z$ is the empty set, we will omit $Z$ from the notation and write $\owusf$ instead.

The limit in \eqref{equation: limit definition wusf} exists and does not depend on the choice of the decreasing exhaustion of $Z$ if $G$ is transient.
See \cite[Theorem~5.1]{BLPS01} for a proof when $Z=\varnothing$; the general case follows from a similar argument.
Note that, if $G$ is recurrent, 
the limit in \eqref{equation: limit definition wusf} can depend on the choice of  the decreasing exhaustion. (However,
only the orientation of $F$ is influenced by this choice; the underlying graph of $F$ remains unchanged!)

%We remark that  oriented wired spanning forests can also be constructed  by using Wilson's method as a disjoint sum of loop-erased random walks,
%see e.g., \cite{Wil96,BLPS01,LP16}.

%We remark that $\owusf$ can also be constructed by using Wilson's method oriented toward infinity. Importantly, 
%we do not remove the orientation of  the edges in the construction.
% We refer to \cite{BLPS01,LP16} for a more detailed discussion on the wired uniform spanning forest.

\section{Proof of the technical lemmas} \label{section: tightness ouroboros}

In this section, we derive three technical lemmas 
that will be used in our proof of  Theorem~\ref{theorem: rotor walk stationarity}.

%Here we list two assumptions that will be satisfied by most (but not all!)  rotor walks in this section:
%\begin{itemize}
%\item[\mylabel{item: NES1}{\textnormal{(NES1)}}] $G$ is a transient graph;
%%\item[\mylabel{item: NES2}{\textnormal{(NES2)}}] 
%%$a$ is a vertex of $G$;
%%\item[\mylabel{item: NES5}{\textnormal{(NES2)}}] 
%%$(Z_R)_{R\geq 0}$ is a decreasing sequence of subsets of $V$  such that   each $V\setminus Z_R$ is finite  and $\bigcap_{R\geq 0}  Z_R= Z$; and 
%%\item[\mylabel{item: NES5}{\textnormal{(NES3)}}]  
% \item[\mylabel{item: NES3}{\textnormal{(NES2)}}] 
%$Z$ is a finite nonempty subset of $V$; 
%\end{itemize}

%One consequence of \ref{item: NES5} and \ref{item: NES3} is that, for any initial rotor configuration $\rho$,  
%the rotor walk $(X_t(a,Z_R;\rho))_{t \geq 0}$ terminates in finite time.
%This in turn implies that  the final rotor configuration $\sigma(a,Z_R;\rho)$ is always well defined.

\subsection{The first technical lemma}
We need the following notations to state this lemma.

\begin{definition}[Odometer]
	We denote by $u(a,Z;\rho)(x)$ the number of visits to $x\in V$ by the rotor walk with initial location $a$, sink $Z$, and initial rotor configuration $\rho$, i.e.,  
	\[u(a,Z;\rho)(x) \ := \  |\{t \geq 0 \mid  X_t(a,Z;\rho)=x \}|.   \]
\end{definition}
%That is to say, $u(a,Z;\rho)(x)$ is the number of times the given rotor walk visits $x$. 
For any  finite subset $W$ of $V$, we write
\[ u(a,Z;\rho)(W) \ := \  \sum_{x \in W} u(a,Z;\rho)(x)\.. \]
For any  $K\geq 0$, we write
\[ \Cc_{K,W}(a,Z) \ := \ \{\,  \rho \, \mid \,  u(a,Z;\rho)(W)< K  \, \}, \]
the set of rotor configurations  for which the corresponding rotor walk visits $W$  strictly less than $K$ times.

Recall the definition of  decreasing exhaustion $(Z_R)_{R\geq 0}$ from Section~\ref{subsection: wired spanning forest}.
We now present the main lemma of this subsection, which gives a probabilistic  upper bound for the odometer. 
\begin{lemma}\label{lemma: bounded number of visits to W}
	Suppose that $G$ is a transient graph.
	Then, for any $\varepsilon >0$ and any finite  $W \subseteq V$, there exists $K:=K(\varepsilon, G, a,W,Z)$ such that, for any $R\geq 0$, 
	\[ \Pb[\rho_R \notin \Cc_{K,W}(a,Z_R)] \ \leq \ \frac{\varepsilon}{2}, \]
	where $\rho_R$ is sampled from $\owusf(Z_R)$.
\end{lemma}

We will derive Lemma~\ref{lemma: bounded number of visits to W} as a consequence of the following lemma from~\cite{Cha19}.
We denote  by $\Gc(a,Z)(x)$  the expected number of visits to $x$ by the simple random walk that starts at $x$ and terminates upon hitting $Z$.
%We will use the following lemma from \cite{Cha19} in the proof.
%The next lemma states that, under some assumptions, these two numbers are equal.

\begin{lemma}[{\cite[Proposition~3.4]{Cha19}}]\label{lemma: odometer finite}
	Suppose that $Z$ is a nonempty subset of $V$ such that $V \setminus Z$ is finite.
	Then, for any $x \in V$,
	\[ \Eb[u(a,Z;\rho)(x)] \ =  \ \Gc(a,Z)(x), \]
	where $\rho$ is sampled from $\owusf(Z)$. \qed
\end{lemma}

%The lemma above will be used in our proof of Theorem~\ref{theorem: rotor walk stationarity} in the following form.
\begin{proof}[Proof of Lemma~\ref{lemma: bounded number of visits to W}]
	We have for any $K\geq 0$ that
	\begin{equation}\label{equation: Markov inequality}
	\begin{split}
	\Pb[u(a,Z_R; \rho_R)(W)\geq K] \quad \leq& \quad \frac{\Eb[u(a,Z_R; \rho_R)(W)]}{K} 
	\quad = \quad   \frac{\Gc(a,Z_R)(W)}{K}, 
	\end{split}
	\end{equation}
	where the inequality is due to Markov's inequality, and the equality is due to Lemma~\ref{lemma: odometer finite}.
	Now note that $\Gc(a,Z_R)(W)$ increases  to $\Gc(a, Z)(W)$ as $R \to \infty$ by \eqref{equation: DecEx}.
	The latter is a finite number as $G$ is transient.
	
	Now choose $K:=\frac{2 \Gc(a,Z)(W)}{\varepsilon}$.
	Substituting this value of $K$ into \eqref{equation: Markov inequality}, 
	\begin{equation*}
	\begin{split}
	\Pb[u(a,Z_R; \rho_R)(W)\geq K] \quad \leq  \quad \frac{\varepsilon}{2} \frac{\Gc(a,Z_R)(W)}{\Gc(a,Z)(W)} \quad \leq \quad \frac{\varepsilon}{2}.
	\end{split}
	\end{equation*}
	% That is to say, a typical  rotor walk $(X_t(a,Z_R,\rho_R))_{t \geq 0}$ will visit  $W$ strictly less than $K$ times.
	This proves the claim.
\end{proof}

\subsection{The second technical lemma}
We need the following notations to state this lemma.
For any $x, y \in V$,
we denote by $d_G(x,y)$ the graph distance between $x$ and $y$ in $G$.
For any $W \subseteq V$, we write $d_G(x,W) \ := \ \min_{y \in W} \{d_G(x,y) \}$\..
\begin{definition}\label{definition: weird rotor configurations}
	For any $r\geq 0$ and any $W \subseteq V$, 
	we denote by $\Ec_{r,W}(a,Z)$   the set of rotor configurations $\rho$ for which  there exists $s_2 > s_1 \geq 0$ so that
	
	\[d_G(W,X_{s_1}(a,Z;\rho)) \ \geq \ r\.; \qquad \text{and} \qquad X_{s_2}(a,Z;\rho) \. \in \. W. \qedhere \]
	%
	%
	%\begin{itemize}
	%\item there exists $t_1\geq 0$  so that 
	%\[d_G(W,X_{t_1}(a,Z;\rho))\geq r; \quad \text{and}\]
	%\item there exists $t_2 > t_1$ so that
	%\[X_{t_2}(a,Z;\rho) \in W.  \qedhere\]
	%\end{itemize}
\end{definition}

That is to say, consider the rotor walk with initial location $a$, sink $Z$, and initial rotor configuration $\rho$.
Then $\rho$ is contained in $\Ec_{r,W}(a,Z)$  if this  rotor walk ends up visiting $W$ some time after visiting a point that is of (graph) distance $r$ from $W$. 
Understanding the event $\Ec_{r,W}$ will be crucial in proving Theorem~\ref{theorem: rotor walk stationarity}.

%Our fourth  ingredient is the following lemma.
\begin{definition}
	Suppose that $Z$ is a nonempty subset of $V$ such that $V \setminus Z$ is finite, 
	$W$ is any  subset of $V$,  and  $\overline{Z}:=W \cup Z$.
	For any (not necessarily distinct) vertices $a_0,a_1,a_2,\ldots$ of $G$ and any rotor configuration $\rho$, 
	we denote by $\xi_i:=\xi_{i}(a_0,a_1,\ldots, a_{i-1}, \overline{Z};\rho)$  the rotor configuration
	\begin{equation}\label{equation: definition xi}
	\xi_i \ := \ \begin{cases} \rho & \text{ if }i=0;\\ \sigma(a_{i-1},\overline{Z};\xi_{i-1})  & \text{ if } i \geq 1,\end{cases} 
	\end{equation}
	where each $\xi_i$ is  well defined because the corresponding  rotor walk terminates in a finite time by the assumption on $Z$.
\end{definition}
Described in words, we let multiple walkers in turn perform a rotor walk 
that terminates upon hitting the enlarged sink $\overline{Z}$, 
with the $i$-th walker starting its walk from $a_i$,
and with  
$\xi_i$ being  the final rotor configuration after the $i$-th walker finishes its walk.  
%(Note that $\xi_i$ is  well defined since each rotor walk terminates in a finite time by \ref{item: DE}.)

Recall that, for any subset $W$ of $V$, the {neighbor set} $N(W)$ of $W$  is the set of vertices of $G$ that are adjacent to a vertex in $W$.  
We now present the main lemma of this subsection. 

%We denote by $\partial W$ the set 
%\[ \partial W:= \{x \in V \ \mid \ d_G(x,W)=1 \text{ and } x\notin W  \}. \]
\begin{lemma}\label{lemma: rho change to xi}
	Suppose that $Z$ is a nonempty subset of $V$ such that $V \setminus Z$ is finite, 
	$W$ is any  subset of $V$,  and  $\overline{Z}:=W \cup Z$.
	Suppose that $\rho$ is contained in $\Ec_{r,W}(a,Z)$ for some $r\geq 0$.
	Then   there exists $i \in \{0,1,\ldots, u(a,Z;\rho)-1\}$ and  $a_1,\ldots, a_{i} \in N(W)$ such that 
	\[ \xi_{i}(a,a_1,\ldots, a_{i-1}, \overline{Z};\rho)  \text{ \. is contained in \. }  \Ec_{r,W}(a_{i},\overline{Z}). \]
\end{lemma}
Intuitively speaking, Lemma~\ref{lemma: rho change to xi} lets  us  enlarge the sink of the rotor walk to include $W$ when checking if the event $\Ec_{r,W}$ occurs.
% modify  the original rotor walk to a new rotor walk with sink enlarged to include $W$.
This will simplify the analysis of the event $\Ec_{r,W}$ in Section~\ref{section: proof of main theorem}.

\begin{proof}[Proof of Lemma~\ref{lemma: rho change to xi}]
	Let $k:=u(a,Z;\rho)$, and let 
	$t_1,t_2,\ldots, t_k$ be the times of visits to $W$ by the (original) rotor walk $(X_t(a,Z;\rho))_{t\geq 0}$, i.e.,  
	\begin{align*}
	t_i& \ := \ \begin{cases} -1 & \text{ if } i=0;\\
	\min \{\,  t >t_{i-1}  \, \mid \, X_t(a,Z;\rho)\in W \, \} & \text{ if } i \in \{1,\ldots,k\}.
	\end{cases}
	\end{align*}
	We now make the following choice of $a_0, a_1,\ldots, a_k$:
	\begin{align*}
	a_i \ = \  \begin{cases}
	a & \text{ if } i=0;\\
	X_{t_i+1} & \text{ if } i\in \{1,2,\ldots,k\}.
	\end{cases}
	\end{align*}
	% \[a_0:=a; \qquad a_i:= X_{t_i+1}  \quad (i \in \{1,2,\ldots,k\}). \]
	Note that all $a_1,a_2,\ldots, a_k$ are  contained in $N(W)$ by \eqref{equation: definition rotor walk}, as they are the vertices visited by the walker right after it reaches $W$.  
	%Intuitively, we are emulating the rotor walk with the smaller sink $Z_R$ by using the rotor walk on the larger sink $\overline{Z_R}$
	
	By making this choice,
	it now follows (from induction) that, for any $i \in \{0,1,\ldots,k-1\}$ and any $s \in \{1,2,\ldots, t_{i+1}-t_{i}\}$, 
	\begin{equation}\label{equation: emulation}
	X_{s}(a_{i},\overline{Z};\xi_{i}) \ = \ X_{t_{i}+s}(a,Z;\rho).
	\end{equation}
	
	% we can now relate  $\xi_i$ $(i \in \{1,2, \ldots,k\})$  to the  original  rotor walk $(X_t(a,Z_R;\rho))_{t\geq 0}$ in the following manner.
	%Start  the rotor walk with initial location $a$ and initial rotor configuration $\rho$.
	% Whenever a walker  reaches $W$,
	% we remove this walker; then add a new walker
	%   to  where the  original walker is supposed to be in the next step; and then  continue the rotor walk with the new walker.
	%   
	%The configuration $\xi_i$ is the rotor configuration during the $i$-th visit to $W$ by the modified walk. 
	
	Now suppose that $\rho$ is contained in  $\Ec_{r,W}(a,Z)$.
	This means that the original rotor walk  $(X_t(a,Z;\rho))_{t\geq 0}$ ends up visiting $W$ some time after visiting a point that is at distance $r$ from $W$.
	Suppose that this visit to $W$ is the $i+1$-th visit to $W$
	(note that $i< k$ as there are at most $k$ visits).  
	%Suppose that this event happens right after the $i$-th visit to $W$ (note that $i<k$ as there are at most $k$ visits). 
	%This event should happen in between the $t_i$-th and $t_{i+1}$-th step of the original walk for some $i<k$ (by the definition of $t_i$'s).
	By \eqref{equation: emulation},
	it follows that the same event (i.e., visiting $W$ after visiting a point that is at distance $r$ from $W$) also occurs for the   modified rotor walk $(X_s(a_{i},\overline{Z};\xi_{i}))_{s\geq 0}$, which is equivalent to $\xi_{i} \in \Ec_{r,W}(a_{i},\overline{Z})$.
	This proves the lemma.
\end{proof}

\subsection{The third technical lemma}\label{subsection: the third technical lemma}
%We  require the following notations to state this lemma.
Throughout the rest of this paper,
we will often require $G$ to satisfy the assumption \eqref{equation: 1End}, which we restate here for the convenience of the reader:

\begin{equation}
\text{Every tree in $\owusf(\varnothing)$ has a single  end  a.s..} \tag{1End}
\end{equation}

%\begin{itemize}
%\item[\mylabel{item: SE}{\textnormal{(1End)}}] 
% \text{Every tree in $\owusf(\varnothing)$ has one single end  a.s..} 
%\end{itemize} 

%{\color{red} Include the usual illustration.}

%Recall the definition of  decreasing exhaustion $(Z_R)_{R\geq 0}$ from Section~\ref{subsection: wired spanning forest}.
We now state the main lemma of this subsection. 
\begin{lemma}\label{lemma: tightness ouroboros}
	Suppose that $G$ is a transient graph for which \eqref{equation: 1End} holds, and $Z$ is nonempty and finite.
	Then, for any $\varepsilon>0$, there exists a positive $r:=r(\varepsilon,G,Z)$ such that
	\[ \lim_{R \to \infty}  \Pb[\rho_R \in \Ec_{r,Z}(a,Z_R) ] \quad \leq  \quad \varepsilon, \]
	where  $\rho_R$ is sampled from $\owusf(Z_R)$.
\end{lemma}
(Note  that the choice of the constant $r$ in Lemma~\ref{lemma: tightness ouroboros} does not depend on the initial location $a$.)
Intuitively speaking, Lemma~\ref{lemma: tightness ouroboros} says that
$\Ec_{r,W}$ is   a rare event with the right choice of the sink.

Our proof of  Lemma~\ref{lemma: tightness ouroboros} uses two technical lemmas from \cite{JR08} and \cite{HLM08}; we restate them here for the convenience of the reader.

We denote by $\Dc_r$ the set of rotor configurations with an oriented path that starts at a point at distance $r$ from $Z$ and ends in $Z$, i.e.,
\[ \Dc_r \ := \ \{\, \rho \, \mid  \, \exists \ x \in V, \, \ell \geq 0   \ \text{  s.t. } \ d_G(x,Z)=r  \,  \text{ and } \, \rho^{\ell}(x) \in Z  \, \}. \]

%That is to say, the set $\Dc_r$ consists of rotor configurations
%that have an oriented path that starts at a point of distance $r$ from $Z$ and ends in $Z$.

\begin{lemma}{\cite[Proposition~7.11]{JR08}}\label{lemma: two part spanning tree}
	Suppose that $G$ is a transient graph for which \eqref{equation: 1End} holds, and $Z$ is nonempty and finite.
	Then, for any $\varepsilon>0$, there exists a positive $r:=r(\varepsilon,G,Z)$ such that
	\[ \lim_{R \to \infty}  \Pb[\rho_R \in \Dc_r ] \quad \leq \quad  \varepsilon, \]
	where  $\rho_R$ is sampled from $\owusf(Z_R)$. \qed
\end{lemma}
Intuitively speaking, Lemma~\ref{lemma: two part spanning tree} says that a typical $Z_R$-oriented spanning forest does not have a very long oriented path  that ends in $Z$.
We remark that Lemma~\ref{lemma: two part spanning tree} was stated in \cite{JR08} only for the case when $Z$ is a singleton, but the proof of the general case follows a similar argument.

%\begin{proposition}{\cite[Proposition~7.11]{JR08}}\label{proposition: two part spanning tree}
%Let $G$ be a simple connected graph that is locally finite and transient, let $A$ be a finite subset of $V$ that contains $a$.
%Consider the spanning forest $\widehat{\rho_R}$ sampled from $\owusf(G; A \cup Z_R)$.
%Suppose that  every tree in $\wusf$ has no infinite backward path a.s.,
%Then, for any $\varepsilon >0$, we have for sufficiently large $r:=r(\varepsilon,A)$ that
%\[ \pushQED{\qed}  \lim_{R \to \infty}  \Pb[\text{there is a path in $\widehat{\rho_R}$ from $\partial B_r$ to $A$}  ] \leq \varepsilon. \qedhere \popQED \]
%\end{proposition}

%Recall that we denote by $\sigma(a,\rho)$ the final rotor configuration of a transient rotor walk with initial location $a$ and initial rotor configuration $\rho$.

\begin{lemma}[{\cite[Lemma~3.11]{HLM08}}]\label{lemma: rotor walk stationarity for finite graph}
	Suppose that $Z$ is  a nonempty subset of $V$ such that $V\setminus Z$ is finite. 
	%Consider any rotor walk on $G$ with initial location
	%$a$ and with nonempty sink $Z$. 
	If $\rho$ is
	sampled from $\owusf(Z)$, then  $\sigma(a,Z;\rho)$ 
	follows the law of $\owusf(Z)$. \qed
\end{lemma}
Note that the rotor walk in Lemma~\ref{lemma: rotor walk stationarity for finite graph} is transient a.s. because of the assumption on $Z$, and therefore the corresponding final rotor configuration  is well defined a.s..

%Note that the configuration $\sigma(a,Z;\rho)$ in Lemma~\ref{lemma: rotor walk stationarity for finite graph} is well defined  
%because the corresponding rotor walk always terminates in a finite time by the assumption on $Z$.
%\begin{lemma}\label{lemma: trace of the trajectory}
%Let $G$ be a simple connected graph that is locally finite.
%Consider any rotor walk on $G$ with initial location $a$, with nonempty sink $Z$, and with initial rotor configuration $\rho$.
%If a vertex $x \in V$ is visited before a vertex $z \in Z$, then there exists a directed path from $x$ to $x$ in $\sigma(a,\rho)$.
%\end{lemma}
%\begin{proof}
%The lemma follows from the fact that the directed path is the loop erasure of the trajectory of the walker from $x$ to $z$.
%\end{proof}
%

%We are now ready to present the proof of Lemma~\ref{lemma: tightness ouroboros}

\begin{proof}[Proof of Lemma~\ref{lemma: tightness ouroboros}]
	Suppose that $\rho$ is a rotor configuration such that the rotor walk $(X_t(a,Z_R;\rho))_{t \geq 0}$ visits $Z$ some time after visiting a point $x$ that is at distance $r$ from $Z$.
	By Lemma~\ref{lemma: trace of the trajectory}, there exists an oriented path in $\sigma(a,Z_R;\rho)$ that starts at $x$ and ends in $Z$.
	That is to say,
	
	\begin{equation}\label{equation: tightness ouroboros 1}
	\rho \in \Ec_{r,Z}(a,Z_R) \quad \Longrightarrow \quad  \sigma(a,Z_R;\rho) \in \Dc_r.
	\end{equation}
	
	Let $\rho_R$ be a rotor configuration sampled from $\owusf(Z_R)$.
	Then
	\begin{equation}\label{equation: tightness ouroboros 2}
	\Pb[\rho_R \in \Ec_{r,Z}(a,Z_R) ] \quad \leq \quad \Pb[\sigma(a,Z;\rho_R) \in \Dc_r ] \quad = \quad \Pb[\rho_R \in \Dc_r ],  
	\end{equation}
	where the inequality  is due to \eqref{equation: tightness ouroboros 1} and the equality is due to Lemma~\ref{lemma: rotor walk stationarity for finite graph}.
	The   lemma now follows from \eqref{equation: tightness ouroboros 2} and Lemma~\ref{lemma: two part spanning tree}. 
\end{proof}

\section{Proof of Theorem~\ref{theorem: rotor walk stationarity}}\label{section: proof of main theorem}
In this section we present a proof of Theorem~\ref{theorem: rotor walk stationarity}.
We restate Theorem~\ref{theorem: rotor walk stationarity} here for the convenience of the reader.
(Note that the sink $Z$ in Theorem~\ref{theorem: rotor walk stationarity} is equal to the empty set $\varnothing$.)

\begin{reptheorem}{theorem: rotor walk stationarity}
	Suppose that  $G$ is a transient graph for which \eqref{equation: 1End} holds, and suppose that $\rho$  is sampled from $\owusf(\varnothing)$. 
	Then the final rotor configuration $\sigma(a,\varnothing; \rho)$ follows the law of $\owusf(\varnothing)$. 
\end{reptheorem}
%Note that the rotor walk in Theorem~\ref{theorem: rotor walk stationarity} is transient a.s. (see \cite[Theorem~1.1]{Cha19}), and therefore the corresponding final rotor configuration  is well defined a.s..

We now build toward the proof of Theorem~\ref{theorem: rotor walk stationarity}.
Our first ingredient is the following lemma from \cite{JR08}.
Recall the definition of $\Ec_{r,W}(a,Z)$  from Definition~\ref{definition: weird rotor configurations}.
%The next lemma gives an equivalent condition that can be used to check if the conclusion of Theorem~\ref{theorem: rotor walk stationarity} holds.

\begin{lemma}[{\cite[Theorem~6.2]{Cha19}}]\label{lemma: TFAE stationarity and trivial tail}
	Suppose that $G$ is a transient graph and  $Z=\varnothing$. TFAE:
	\begin{itemize}
		\item The configuration   $\sigma(a,\varnothing;\rho)$ follows the law of $\owusf(\varnothing)$ when $\rho$ is sampled from $\owusf(\varnothing)$;
		\item  For any $\varepsilon>0$ and any  nonempty finite subset $W$ of $V$, there exists a positive  $r:=r(\varepsilon,G,a,W)$ so that
		\begin{equation}\label{equation: stationarity condition to check}
		\lim_{R \to \infty}  \Pb[\rho_R \in \Ec_{r,W}(a,Z_R) ] \quad \leq \quad \varepsilon, \end{equation}
		where  $\rho_R$ is sampled from $\owusf(Z_R)$. \qed
	\end{itemize}
\end{lemma}
Lemma~\ref{lemma: TFAE stationarity and trivial tail} reduces  proving Theorem~\ref{theorem: rotor walk stationarity} to checking that 
the event $\Ec_{r,W}$ is very unlikely to occur.

Our second ingredient is the following technical lemma from \cite{JR08}.

\begin{lemma}[{\cite[Lemma~7.5]{JR08}}]\label{lemma: Radon-Nikodym spanning forest}
	Suppose that $G$ is a transient graph,
	$Z=\varnothing$, and  $W$ is a  nonempty finite subset of $V$.
	%Also suppose that $(Z_R)_{R\geq 0}$ and $(\overline{Z_R})_{R\geq 0}$ satisfies \ref{item: NES5}  for $Z$ and $\overline{Z}$, respectively.
	Then there exists $c(G,W)> 0$ such that,
	for every $R\geq 0$,
	\begin{equation*}
	\frac{\SF(\overline{Z_R})}{\SF({Z_R})} \quad \leq \quad  c(G,W)\..
	\end{equation*}
\end{lemma}

The proof for Lemma~\ref{lemma: Radon-Nikodym spanning forest} here is paraphrased from the (second) proof of Lemma~7.5 in \cite{JR08}, 
and is included here for the sake of completeness.

\begin{proof}[Proof of Lemma~\ref{lemma: Radon-Nikodym spanning forest}]
	We can without loss of generality assume that   $W$ is a singleton $\{x\}$, 
		as the general case will then follow from successive applications of the same argument.
	Let $\widehat{G}$ be the graph obtained from $G$ by adding an extra vertex $y$ and an extra edge $e=\{x,y\}$ connecting $x$ and $y$.
	Let \. $\widehat{Z_R} := Z_R \cup \{y\}$\..
%	Note that the probability in Lemma~\ref{lemma: Radon-Nikodym spanning forest} can be rewritten as 
	We have
	\begin{align}\label{equation: RN 1}
		\frac{\SF(\overline{Z_R})}{\SF({Z_R})} \quad = \quad  \frac{\Pb[(x,y) \in \widehat{F_R}]}{1- \Pb[(x,y) \in \widehat{F_R}]}\.,
	\end{align}
	where \. $\widehat{F_R}$ \. is sampled from $\owusf(\widehat{Z_R})$\..
	
	By using Wilson's algorithm~\cite{Wil96} to sample the $\widehat{Z_R}$-spanning forest  $\widehat{F_R}$,
	the probability 
	\. $\Pb[(x,y) \in \widehat{F_R}]$ \. is the probability that a simple random walk on $G$  started at $x$ first hits $\widehat{F_R}$ through $(x,y)$.
	This implies that
	\begin{equation}\label{equation: RN 2}
	 \lim_{R \to \infty}  \Pb[(x,y) \in \widehat{F_R}]  \quad = \quad 1 - p,   
	\end{equation}
	where $p:=p_{x,y}$ is the probability that the simple random walk on $G$ started at $x$ never visits $y$ (note that $p>0$ since $G$ is transient).	
	Plugging \eqref{equation: RN 2} into \eqref{equation: RN 1}, we get
	\[ \lim_{R \to \infty} \frac{\SF(\overline{Z_R})}{\SF({Z_R})} \quad = \quad  \frac{1}{p}  - 1 \quad < \quad \infty\.,  \]
	and the lemma now follows.
\end{proof}
 
We will use the following consequence of Lemma~\ref{lemma: Radon-Nikodym spanning forest}.
A set $\Bc$ of rotor configurations is \emph{$Z$-invariant} if $\Bc$ depends only on rotors outside of $Z$.
That is,
for any  $\rho,\rho'$,
\[ \rho(x)=\rho'(x) \ \text{ for all } \ x \in V\setminus \{Z\}  \ \text{ and } \ \  \rho \in \Bc \qquad \Longrightarrow \qquad \rho' \in \Bc.    \]
%, for any rotor configurations $\rho,\rho'$ such that $\rho(x)=\rho'(x)$ for all $x \in V \setminus Z$, we have $\rho\in \Bc$ whenever $\rho' \in \Bc$.

\begin{lemma}\label{lemma: Radon-Nikodym derivatives is positive}
	Suppose that $G$ is a transient graph,
	$Z=\varnothing$, and  $W$ is a  nonempty finite subset of $V$.
	%Also suppose that $(Z_R)_{R\geq 0}$ and $(\overline{Z_R})_{R\geq 0}$ satisfies \ref{item: NES5}  for $Z$ and $\overline{Z}$, respectively.
	Then there exists $c:=c(G,W)> 0$ such that,
	for all $R\geq 0$,
	\begin{equation*}
	\Pb[\rho_R \in \Bc_R  ]  \quad \leq  \quad  c  \, \Pb[\overline{\rho_R} \in \Bc_R  ], 
	\end{equation*}
	where $\rho_R$ is sampled from $\owusf(Z_R)$, $\overline{\rho_R}$ is sampled from $\owusf(\overline{Z_R})$, and $\Bc_R$ is an arbitrary $\overline{Z_R}$-invariant set. \qed
\end{lemma}

\begin{proof}
	Let $\Ac_R$ and $\overline{\Ac_R}$ be the set of oriented subgraphs of $G$ given by
	\[ \Ac_R \ := \  \{ F \in \SF(Z_R) \. \mid \. \rho_F \in \Bc_R  \}; \qquad \overline{\Ac_R} \ := \  \{ F \in \SF(\overline{Z_R}) \. \mid \. \rho_F \in \Bc_R  \}\..  \]	
	The  probabilities in Lemma~\ref{lemma: Radon-Nikodym derivatives is positive} can then be rewritten as
	\begin{align}\label{equation: RND 1}
	\Pb[\rho_R \in \Bc_R  ] \ = \ \frac{|\Ac_R|}{|\SF(Z_R)|}\.; \qquad
	\Pb[\overline{\rho_R} \in \Bc_R  ] \ = \ \frac{|\overline{\Ac_R}|}{|\SF(\overline{Z_R})|}\..
	\end{align}

	Let \. $\phi: \Ac_R \to \overline{\Ac_R}$ \. be the map sending $F \in \Ac_R$ to the subgraph $\phi(F)$ obtained from $F$ by deleting edges with source vertex in $W$.
	 Note that $\phi(F)$  is contained in $\overline{\Ac_R}$ since 
	since $\Bc_R$ is a $\overline{Z_R}$-invariant set\..
	Also note that, for every element of $\overline{\Ac_R}$,
	its preimage under $\phi$ contains at most $\prod_{x \in W} \deg(x)$ many elements.
	This implies that 
	\begin{align}\label{equation: RND 2}
		|\Ac_R|  \quad \leq \quad  |\overline{\Ac_R}| \ \prod_{x \in W} \deg(x)\.. 
	\end{align}
	On the other hand, we have by Lemma~\ref{lemma: Radon-Nikodym spanning forest} that there exists $c':=c'(G,W)$ such that
	\begin{equation}\label{equation: RND 3}
	\frac{\SF(\overline{Z_R})}{\SF({Z_R})} \quad \leq \quad  c'\..
	\end{equation}
	Plugging \eqref{equation: RND 3} and \eqref{equation: RND 2} into \eqref{equation: RND 1},
	we get
	\[  \frac{\Pb[\rho_R \in \Bc_R  ]}{\Pb[\overline{\rho_R} \in \Bc_R  ]} \quad \leq \quad  c' \prod_{x \in W} \deg(x)\.,  \]
	for which the lemma follows.
\end{proof}

Intuitively speaking, Lemma~\ref{lemma: Radon-Nikodym derivatives is positive} 
lets us switch the initial rotor configuration with another configuration that will interact nicely with the event $\Ec_{r,W}$.
%Lemma~\ref{lemma: Radon-Nikodym derivatives is positive} 
%is equivalent to  Lemma~7.5 in \cite{JR08}, and we present their proof here, in the language of this paper, for the sake of completeness.

%says that the density of the measure $\owusf(Z_R)$ with respect to the measure $\owusf(\overline{Z_R})$ is strictly bounded from below by a positive constant.
%Intuitively speaking, Lemma~\ref{lemma: Radon-Nikodym derivatives is positive} will allow us to substitute the initial rotor configuration of a rotor walk from $\rho_R$ to $\overline{\rho_R}$.

%We remark that  Lemma~\ref{lemma: Radon-Nikodym derivatives is positive} was stated in \cite{JR08} only for the case when $W$ is a singleton, but the proof of the general case follows a similar argument.

Our third, fourth, and fifth ingredients are Lemma~\ref{lemma: bounded number of visits to W}, Lemma~\ref{lemma: rho change to xi}, and Lemma~\ref{lemma: tightness ouroboros}, respectively.
We are now ready to present our proof of Theorem~\ref{theorem: rotor walk stationarity}.
\begin{proof}[Proof of Theorem~\ref{theorem: rotor walk stationarity}]
	Let $Z=\varnothing$, let $\varepsilon>0$, and let $W$ be an arbitrary nonempty finite subset of $V$.
	It suffices to check that \eqref{equation: stationarity condition to check} holds under the assumptions of Theorem~\ref{theorem: rotor walk stationarity}.

	Let $\rho_R$ be sampled from $\owusf$.
	Let $K:=K(\varepsilon, G, a,W)$ be as in Lemma~\ref{lemma: bounded number of visits to W}.
	We have for any $r \geq 0$ that
	\begin{equation}\label{equation: proof of main theorem 1}
	\begin{split}
	& \Pb[\rho_R \in \Ec_{r,W}(a,Z_R) ] \\
	\leq \  & \Pb[\rho_R \notin \Cc_{K,W}(a,Z_R) ] \ + \ \Pb[\rho_R \in \Ec_{r,W}(a,Z_R) \cap \Cc_{K,W}(a,Z_R)]\\
	\leq \ &  \frac{\varepsilon}{2} \ +\  \Pb[\rho_R \in \Ec_{r,W}(a,Z_R) \cap \Cc_{K,W}(a,Z_R)],
	\end{split}
	\end{equation}
	where the second inequality is due to Lemma~\ref{lemma: bounded number of visits to W}.
	
	Write $\overline{Z_R}:= W \cup Z_R$. 
	We have by   Lemma~\ref{lemma: rho change to xi}
	that 
	\begin{equation}\label{equation: proof of main theorem 2}
	\begin{split}
	& \Pb[\rho_R \in \Ec_{r,W}(a,Z_R) \cap \Cc_{K,W}(a,Z_R)] \\
	\leq \ & \sum_{i=0}^{K-1} \sum_{a_1,\ldots, a_{i} \in N(W) } \Pb[ \xi_{i}(a,a_1,\ldots, a_{i-1}, \overline{Z_R};\rho_R) \in \Ec_{r,W}(a_{i},\overline{Z_R})].
	\end{split}
	\end{equation}  
	
	Let $c:=c(G,W)$ be as in Lemma~\ref{lemma: Radon-Nikodym derivatives is positive}.
	Applying Lemma~\ref{lemma: Radon-Nikodym derivatives is positive} with
	$\Bc_R$ being  
	\[ \Bc_R  \ := \  \Bc_R(a,a_1,\ldots, a_{i-1}) \ := \ \big \{ \rho \. \mid \. \xi_{i}(a,a_1,,\ldots, a_{i-1}, \overline{Z_R};\rho) \in \Ec_{r,W}(a_i,\overline{Z_R}) \big\}\.,   \]
	we get 
	\begin{equation}\label{equation: proof of main theorem 3}
	\begin{split}
	& \Pb[ \xi_{i}(a,a_1,\ldots, a_{i-1}, \overline{Z_R};\rho_R) \in \Ec_{r,W}(a_i,\overline{Z_R})] \\
	\leq &  \ c  \, \Pb[ \xi_{i}(a,a_1,\ldots, a_{i-1}, \overline{Z_R};\overline{\rho_R}) \in \Ec_{r,W}(a_i,\overline{Z_R})],
	\end{split}
	\end{equation}  
	where $\overline{\rho_R}$ is sampled from $\owusf(\overline{Z_R})$.
	
	Now note that $\xi_{i}(a,a_1,\ldots, a_{i-1}, \overline{Z_R}; \overline{\rho_R})$ follows the law of $\owusf(\overline{Z_R})$
	by \eqref{equation: definition xi} and Lemma~\ref{lemma: rotor walk stationarity for finite graph}, so we have 
	\begin{equation}\label{equation: proof of main theorem 4}
	\begin{split}
	& \Pb[ \xi_{i}(a,a_1,\ldots, a_{i-1}, \overline{Z_R};\overline{\rho_R}) \in \Ec_{r,W}(a_i,\overline{Z_R})] \quad = \quad \Pb[ \overline{\rho_R} \in \Ec_{r,W}(a_i,\overline{Z_R})].
	\end{split}
	\end{equation}  
	
	Now let $L:=2\, c \, \sum_{i=0}^{K-1}|N(W)|^i$, and  choose  
	$r=r(\varepsilon/L,G,W)$ to be as in Lemma~\ref{lemma: tightness ouroboros}.
	It follows from Lemma~\ref{lemma: tightness ouroboros} that
	\begin{equation}\label{equation: proof of main theorem 5}
	\begin{split}
	& \Pb[ \overline{\rho_R} \in \Ec_{r,W}(a_i,\overline{Z_R})] 
	 \quad  \leq  \quad  \frac{\varepsilon}{L}.
	\end{split}
	\end{equation}

	Combining \eqref{equation: proof of main theorem 1}, \eqref{equation: proof of main theorem 2}, \eqref{equation: proof of main theorem 3},  \eqref{equation: proof of main theorem 4},  and \eqref{equation: proof of main theorem 5} together, we get
	\begin{equation}\label{equation: proof of main theorem 6}
	\begin{split}
	\Pb[\rho_R \in \Ec_{r,W}(a,Z_R) ] \quad \leq &\quad  \frac{\varepsilon}{2}+ c \sum_{i=0}^{K-1} \sum_{a_1,\ldots, a_i \in N(W)}  \frac{\varepsilon}{L}  \quad = \quad  \varepsilon.
	\end{split}
	\end{equation}
	The theorem now follows from \eqref{equation: proof of main theorem 6} and Lemma~\ref{lemma: TFAE stationarity and trivial tail}.
\end{proof}

Finally, we remark that the following strengthened version of Theorem~\ref{theorem: rotor walk stationarity} can be proved verbatim.

\begin{theorem}
	Suppose that  $G$ is a transient graph for which \eqref{equation: 1End} holds and $Z$ is finite.
	Then the final rotor configuration $\sigma(a,Z; \rho)$ follows the law of $\owusf(Z)$ when $\rho$ is sampled from $\owusf(Z)$. \qed
\end{theorem}

\section{Concluding remarks}\label{section: open problems}
We conclude with a few natural questions:
\subsection{} Is $\owusf$  the unique infinite-step stationary measure for  rotor walk on   $\Zb^d$ for $d\geq 3$? 
Does there exist an infinite-step stationary measure for  rotor walk on $\Zb^2$?

\subsection{}	%Is the converse of Theorem~\ref{theorem: rotor walk stationarity} true?
%That is, 
Is  \eqref{equation: 1End} a necessary condition for $\owusf$ to be an infinite-step stationary measure for  rotor walk?

\subsection{} Fix a local mechanism for every vertex in $\Zb^d$,
and choose the initial rotor $\rho(x)$ for $x \in \Zb^d$ independently and uniformly at random  from the outgoing edges of $x$.
What is the escape rate of this rotor walk?

%
%\begin{enumerate}[label=(\arabic*)]
%	\item 
%	\item 
%
%	
%
%\end{enumerate}

\section*{Acknowledgement}
The author would like to thank Lionel Levine and Yuval Peres for their advising throughout the whole project, Ander Holroyd  for inspiring discussions, and Yuwen Wang for proofreading.
The author would also like to thank the anonymous referee for their careful review and insightful suggestions.
Part of this work was done when the author was visiting the Theory Group at Microsoft Research, Redmond, 
and when the author was a graduate student at Cornell University.

%%%%%%%%%%%%%%%%%%%%%%%%%%%%%%%%%% Bibliography
\vskip.9cm

\bibliographystyle{amsalpha}

\end{document}